\documentclass[a4paper,12pt]{amsart}

\usepackage[headings]{fullpage}

\usepackage{amsfonts,graphics,amsmath,mathrsfs,amsthm,amscd,amssymb,latexsym,euscript,enumerate}
\usepackage{epsfig}
\usepackage{flafter}
\usepackage[all,cmtip,line]{xy}
\usepackage{array}
\usepackage[english]{babel}
\usepackage{overpic}
\usepackage{subfig}
\usepackage{multirow}
\usepackage{microtype}
\usepackage{wrapfig}
\usepackage{longtable}
\usepackage{supertabular}
\usepackage[shortlabels]{enumitem}
\usepackage{tikz}
\usetikzlibrary{positioning}
\usepackage{tikz-cd}
\usepackage{float}
\allowdisplaybreaks

\usepackage[dvipsnames,svgnames,table]{xcolor}
\usepackage{graphicx}
\usepackage{hyperref}
\hypersetup{
    colorlinks=true,
    linkcolor=blue,
    citecolor=blue,
    filecolor=blue,
    urlcolor=blue
}
\usepackage{cleveref}

\newtheorem{theorem}{Theorem}[section]
\newtheorem{lemma}[theorem]{Lemma}
\newtheorem{proposition}[theorem]{Proposition}

\newtheorem*{theorem*}{Theorem}

\theoremstyle{plain}
\newtheorem{corollary}[theorem]{Corollary}

\theoremstyle{definition}
\newtheorem{definition}[theorem]{Definition}
\newtheorem{definition-lemma}[theorem]{Definition-Lemma}

\numberwithin{equation}{section}

\newcommand{\R}{\mathbb{R}}
\newcommand{\Z}{\mathbb{Z}}
\newcommand{\N}{\mathbb{N}}
\newcommand{\Q}{\mathbb{Q}}
\newcommand{\OO}{\mathcal{O}}

\DeclareMathOperator{\ord}{ord}

\DeclareMathOperator{\Val}{Val}

\DeclareMathOperator{\Nef}{Nef}

\DeclareMathOperator{\Eff}{Eff}

\DeclareMathOperator{\Bigdiv}{Big}
\DeclareMathOperator{\Div}{Div}

\def\Pic{\operatorname{Pic}}

\def\Spec{\operatorname{Spec}}

\def\mult{\operatorname{mult}}

\def\Mov{\operatorname{Mov}}

\def\Supp{\operatorname{Supp}}
\def\Exc{\operatorname{Exc}}

\def\lct{\operatorname{lct}}

\usepackage{mathtools}

\DeclarePairedDelimiterX{\norm}[1]{\lVert}{\rVert}{#1}

\newcommand{\floor}[1]{\left\lfloor #1 \right\rfloor}
\newcommand{\ceil}[1]{\left\lceil #1 \right \rceil}

\title[]{Coupled Pklt Tuples and Varieties of Pklt Type}

\author[D. Kim]{Donghyeon Kim}
\author[D.-W. Lee]{Dae-Won Lee}
\address[Donghyeon Kim]{Department of Mathematics, Yonsei University, 50 Yonsei-ro, Seodaemun-gu, Seoul 03722, Republic of Korea}
\email{narimial0@gmail.com, whatisthat@yonsei.ac.kr}
\address[Dae-Won Lee]{School of Mathematics, Korea Institute for Advanced Study, 85 Hoegiro, Dongdaemun-gu, Seoul 02455, Republic of Korea}
\email{daewonlee@kias.re.kr}

\date{\today}

\begin{document}

\begin{abstract}
We introduce asymptotic multiplier ideal sheaves and log canonical thresholds associated with tuples of pseudoeffective divisors on a projective klt pair. We prove that the threshold of a coupled potentially klt tuple is computed by a quasi-monomial valuation. For varieties of potentially klt type, we prove that every big divisor admits a birational Zariski decomposition with semiample positive part. We also prove finite generation of multisection rings of big divisors and give a criterion for a variety of potentially klt type to be a Mori dream space.
\end{abstract}

\maketitle



\section{Introduction}

The concept of a potentially klt pair was introduced in \cite{CP} in the context of studying the potentially non-klt locus and Fano type varieties. The potential log discrepancy measures the singularities contributed by the negative part of the anticanonical divisor via Nakayama’s asymptotic order. This viewpoint has recently been developed further in several directions. In \cite{CJK}, adjoint asymptotic multiplier ideal sheaves for potential triples were introduced, and a Nadel type vanishing theorem was established. Moreover, given a potential triple $(X,\Delta,D)$, when $D$ admits a birational Zariski decomposition, one can associate a generalized pair as in \cite{CJL2}. Finally, \cite[Theorem 1.1 and Corollary 1.3]{CJKL} showed that the log canonical threshold of a potentially klt triple is determined by a quasi-monomial valuation, and used this valuative characterization in developing the anticanonical minimal model program.

The first author showed that Fano type is generically invariant provided the anticanonical volume remains constant on a Zariski dense subset of the base \cite[Theorem 1.2]{Kim25}. In addition, the present authors proved that geometric generic fibers are potentially klt for fibrations whose closed fibers are uniformly of log Calabi--Yau type \cite[Theorem 1.2]{KL25}. These statements concern the variation of volumes and asymptotic orders in families, see \cite[Theorem 1.1]{Jun26}. In this article, we extend the valuative framework to tuples.

Recall that a projective pair $(X,\Delta)$ with $-(K_X+\Delta)$ pseudoeffective is \emph{potentially klt}, or \emph{pklt} if there exists $\varepsilon>0$ such that $A_{X,\Delta}(E)-\sigma_E(-(K_X+\Delta))\geq \varepsilon$ for every prime divisor $E$ over $X$.

Let $(X,\Delta)$ be a projective klt pair, and let \(\mathbf{D}=(D_1,\dots,D_r)\) be a tuple of pseudoeffective \(\Q\)-Cartier \(\Q\)-divisors on $X$. For a valuation \(\nu\in \Val_X\), set $\sigma_{\mathbf{D}}(\nu)\coloneqq \sum_{i=1}^r \sigma_{\nu}(D_i)$. We say that \((X,\Delta,\mathbf{D})\) is \emph{coupled potentially klt}, or \emph{coupled pklt} if there exists \(\varepsilon>0\) such that $A_{X,\Delta}(E)-\sigma_{\mathbf{D}}(E)\geq \varepsilon$ for every prime divisor \(E\) over \(X\). We define the \emph{log canonical threshold} as
\[
\lct_{\sigma}(X,\Delta,\mathbf{D})\coloneqq \inf_{\nu\in \Val^{\ast}_X}\frac{A_{X,\Delta}(\nu)}{\sigma_{\mathbf{D}}(\nu)}.
\]
The function $\sigma_{\mathbf{D}}$ is the sum of the asymptotic orders of the divisors in the tuple. In general, $\sigma_{\mathbf{D}}$ does not coincide with the asymptotic order of their sum.

The following is the first main result.

\begin{theorem} \label{thrm:1}
Let $(X,\Delta)$ be a projective klt pair, and let $\mathbf D=(D_1,\cdots,D_r)$ be a tuple of pseudoeffective $\Q$-Cartier $\Q$-divisors on $X$. If \((X,\Delta,\mathbf D)\) is coupled pklt, then there exists a quasi-monomial valuation $\omega\in \mathrm{Val}^*_X$ that computes $\mathrm{lct}_{\sigma}(X,\Delta,\mathbf D)$.
\end{theorem}

The proof follows the strategy of \cite[Theorem 1.1]{CJKL}, together with the continuity  Lemma \ref{lem:div} and the concavity of $\sigma_{\mathbf D}$ on cones of quasi-monomial valuations.

\begin{corollary}\label{coro}
Let $(X,\Delta)$ be a projective klt pair, and let $\mathbf D=(D_1,\cdots,D_r)$ be a tuple of pseudoeffective $\Q$-Cartier $\Q$-divisors on $X$. Then, $(X,\Delta,\mathbf D)$ is coupled pklt if and only if $\lct_{\sigma}(X,\Delta,\mathbf D)>1$.
\end{corollary}

The next corollary extends the generic fiber results in \cite[Theorem 1.2]{Kim25}.

\begin{corollary}\label{coro2}
Let $(X,\Delta)\to S$ be a projective flat family of klt pairs over an integral variety $S$, with $K_{X/S}+\Delta$ $\Q$-Cartier. Let $\mathbf D=(D_1,\cdots,D_r)$ be a tuple of big $\Q$-Cartier $\Q$-divisors on $X$. Suppose that there exists a Zariski dense subset $S'\subseteq S$ of closed points such that, for each $i$, the volume $\mathrm{vol}(D_{i,s})$ is constant for $s\in S'$. Let $\eta\in S$ be the generic point. If $(X_s,\Delta_s,\mathbf D_s)$ is coupled pklt for every $s\in S'$, then the geometric generic fiber $(X_{\overline{\eta}},\Delta_{\overline{\eta}},\mathbf D_{\overline{\eta}})$ is also coupled pklt.
\end{corollary}

The remaining results deal with varieties of pklt type. A normal projective variety \(X\) is said to be of \emph{potentially klt type} (\emph{pklt type}) if there is an effective \(\Q\)-divisor \(\Delta\) for which \((X,\Delta)\) forms a potentially klt pair. As a first consequence, any big $\Q$-Cartier divisor on a variety of pklt type has a birational Zariski decomposition whose positive part is semiample.

\begin{theorem}\label{thm:BZD}
    Let $X$ be a projective variety of pklt type. Then every big $\Q$-Cartier $\Q$-divisor on $X$ admits a birational Zariski decomposition with semiample positive part.
\end{theorem}

On a variety of pklt type, the log canonical threshold of a klt pair with a tuple of big divisors is computed by a divisorial valuation.

\begin{corollary}\label{cor:divisorial}
    Let \(X\) be a projective \(\Q\)-factorial variety of pklt type, \((X,\Delta)\) be a klt pair with $\Delta$ a $\Q$-divisor, and \(\mathbf{D}=(D_1,\dots,D_r)\) be a tuple of big \(\Q\)-Cartier \(\Q\)-divisors. Then
    \[
    \lct_{\sigma}(X,\Delta,\mathbf{D})\in \Q_{>0}\cup \{+\infty\}.
    \]
    Moreover, if $\lct_{\sigma}(X,\Delta,\mathbf{D})$ is finite, then the threshold is computed by a divisorial valuation.
\end{corollary}

The multisection ring is finitely generated for any tuple of big divisors on a variety of pklt type.

\begin{theorem}\label{thm:fg}
Let $X$ be a projective variety of pklt type, and let $D_1,\dots,D_r$ be big $\Q$-Cartier \(\Q\)-divisors on $X$. Then the multisection ring
\[
R(X;D_1,\dots,D_r)
\coloneqq
\bigoplus_{(m_1,\dots,m_r)\in\N^r}
H^0\!\left(X,\sum_{i=1}^r m_iD_i\right)
\]
is finitely generated.
\end{theorem}

The last application concerns the boundary of the movable cone.

\begin{theorem}\label{thm:Mov cone}
Let $X$ be a projective $\Q$-factorial variety of pklt type such that $\Pic(X)_{\Q}=N^1(X)_{\Q}$. Assume that
\[
\overline{\Mov}(X)\setminus\{0\}\subset\Bigdiv(X).
\]
Then $X$ is a Mori dream space. 
\end{theorem}

The rest of this paper is organized as follows. In Section \ref{sect-2}, we recall the definitions and properties of valuations, asymptotic orders, potentially klt pairs, birational Zariski decompositions, and Mori dream spaces. In Section \ref{sect-3}, we develop coupled asymptotic and diminished multiplier ideals and prove Theorem \ref{thrm:1} and Corollary \ref{coro}. In Subsection \ref{subsect-ggf}, we prove Corollary \ref{coro2}. The applications to varieties of pklt type are proved in Subsection \ref{subsect-appl}.

\section*{Acknowledgement}
The authors are partially supported by Samsung Science and Technology Foundation under Project Number SSTF-BA2302-03. The first author was supported by Basic Science Research Program through the National Research Foundation of Korea (NRF) funded by the Ministry of Education (No. RS-2026-25555589), and the second author is partially supported by the National Research Foundation of Korea (NRF) (No. RS-2025-00513064 and 2021R1A6A1A10039823), and by a KIAS Individual Grant (MG111201) at Korea Institute for Advanced Study

\section{Preliminaries}\label{sect-2}
Throughout the paper, we work over an algebraically closed field $k$ of characteristic zero. Varieties are normal and projective unless stated otherwise. A pair $(X,\Delta)$ consists of a normal projective variety $X$ and an effective $\Q$-divisor $\Delta$ such that $K_X+\Delta$ is $\Q$-Cartier.

\subsection{Valuations and asymptotic orders}
Let \(X\) be a normal variety with function field \(K\). We write \(\Val_X\) for the set of (real) valuations on \(K\) centered on \(X\), and \(\Val_X^{\ast}\coloneqq \Val_X\setminus \{0\}\). If \(E\) is a prime divisor on a birational model of \(X\), then \(\ord_E\) denotes the associated divisorial valuation. We refer to \cite{JM12} for quasi-monomial valuations and the topology on \(\Val_X\). 

Let \((X,\Delta)\) be a klt pair. The log discrepancy extends to a lower semicontinuous, homogeneous function $A_{X,\Delta}\colon \Val_X\longrightarrow [0,+\infty]$. Let $(Y,E)$ be a log smooth model of $(X,\Delta)$ and let $\eta$ be the generic point of a stratum of $E$. On $\mathrm{QM}_\eta(Y,E)$, associated with a log smooth model $(Y,E)$ of $(X,\Delta)$, $A_{X,\Delta}$ is the linear function $\sum_{j=1}^q \alpha_j A_{X,\Delta}(E_j)$.

Let \(D\) be a pseudoeffective \(\Q\)-Cartier \(\Q\)-divisor. For a big divisor $D$, set $\nu(|D|)\coloneqq \inf_m \frac{\nu(\mathfrak{b}(|mD|))}{m}$, where $mD$ is Cartier and $|mD|\neq \emptyset$. For an ample Cartier divisor $A$, set $\sigma_{\nu}(D)\coloneqq \lim_{t\to 0^{+}}\nu(|D+tA|)$, taking rational $t>0$. The limit is independent of $A$ and equals $\nu(|D|)$ when $D$ is big, see \cite[III.1.5]{Nak}. The function \(D\mapsto \sigma_{\nu}(D)\) is homogeneous and subadditive on the pseudoeffective cone, and it is continuous on the big cone. 

For a point $x\in X$, write $X_x\coloneqq \Spec \OO_{X,x}$, let $\Delta_x$ be the induced boundary, and let $\mathfrak{m}_x$ be the maximal ideal of $\OO_{X,x}$. The notation $\Val_{X,x}$ means that the center is exactly $x$. For a graded sequence $\mathfrak{a}_{\bullet}$, we use $\mathfrak{a}_0=\OO_X$ and $\nu(\mathfrak{a}_{\bullet})=\inf_{m\geq 1} \frac{\nu(\mathfrak{a}_m)}{m}$.

\subsection{Potentially klt pairs and coupled tuples}
Let \((X,\Delta)\) be a pair such that \(-(K_X+\Delta)\) is pseudoeffective. We say that $(X,\Delta)$ is \emph{potentially klt (pklt)} if there exists $\varepsilon > 0$ such that
\[
A_{X,\Delta}(E)-\sigma_E(-(K_X+\Delta))\geq \varepsilon,
\]
where \(E\) runs through all prime divisors over \(X\).

\begin{definition}
    A normal projective variety \(X\) is of \emph{pklt type} if there exists an effective \(\Q\)-divisor \(\Delta\) on \(X\) such that \((X,\Delta)\) is pklt.
\end{definition}

\begin{definition}
    Let \((X,\Delta)\) be a projective klt pair, and let \(\mathbf{D}=(D_1,\dots,D_r)\) be a tuple of pseudoeffective \(\Q\)-Cartier \(\Q\)-divisors on \(X\). For \(\nu\in \Val_X\), set $\sigma_{\mathbf{D}}(\nu)\coloneqq \sum_{i=1}^r \sigma_{\nu}(D_i)$. We define the \emph{log canonical threshold} of \((X,\Delta)\) with respect to \(\mathbf{D}\) as 
    \[
    \lct_{\sigma}(X,\Delta,\mathbf{D})\coloneqq \inf_{\nu\in \Val_X^{\ast}}\frac{A_{X,\Delta}(\nu)}{\sigma_{\mathbf{D}}(\nu)}.
    \]
    The tuple \((X,\Delta,\mathbf{D})\) is \emph{coupled pklt} if there exists \(\varepsilon>0\) such that $A_{X,\Delta}(E)-\sigma_{\mathbf{D}}(E)\geq \varepsilon$
    for every prime divisor \(E\) over \(X\).
\end{definition}

The following observation will be used repeatedly.
\begin{lemma}\label{lem:sup}
    Let \((X,\Delta)\) be a projective klt pair, and let \(D\) be a big \(\Q\)-Cartier \(\Q\)-divisor. Then
    \[
    \sup_{\nu\in \Val_X^{\ast}}\frac{\sigma_{\nu}(D)}{A_{X,\Delta}(\nu)}<+\infty.
    \]
\end{lemma}
\begin{proof}
    Choose an effective \(\Q\)-divisor \(B\sim_{\Q} D\). Since \((X,\Delta)\) is klt, there exists \(t>0\) such that \((X,\Delta+tB)\) is klt. For every \(\nu\in \Val_X^{\ast}\), one has \(t\nu(B)<A_{X,\Delta}(\nu)\). Since \(\sigma_{\nu}(D)\leq \nu(B)\), the assertion follows.
\end{proof}

\begin{lemma}\label{lem-logdiscr}
    Let $(X,\Delta)$ be a klt pair, and let $m$ be a positive integer such that $m(K_X+\Delta)$ is Cartier. Then $A_{X,\Delta}(E)\geq \frac{1}{m}$ for every prime divisor $E$ over $X$.
\end{lemma}
\begin{proof}
    Since $mA_{X,\Delta}(E)$ is a positive integer for every prime divisor $E$ over $X$, we obtain the assertion.
\end{proof}

\subsection{Birational Zariski decompositions and multisection rings}
\begin{definition}
    Let \(D\) be a big \(\Q\)-Cartier \(\Q\)-divisor on \(X\). A birational Zariski decomposition of \(D\) with semiample positive part consists of a projective birational morphism \(f\colon Y\to X\) and a decomposition $f^{\ast}D=P+N$, where \(P\) is a semiample \(\Q\)-Cartier \(\Q\)-divisor and \(N\geq 0\), such that the natural inclusion
    \[
    H^0(Y,mP)\hookrightarrow H^0(Y,mf^{\ast}D)
    \]
    is an isomorphism for every sufficiently divisible positive integer \(m\).
\end{definition}

For a $\Q$-Cartier $\Q$-divisor $D$, we set $H^0(X,D)\coloneqq \{0\}\cup\{f\in k(X)^{\ast}\mid \mathrm{div}(f)+D\ge 0\}$. For $D_1,\dots,D_r\in\Div_{\Q}(X)$, the associated multisection ring is
\[
R(X;D_1,\dots,D_r)
\coloneqq
\bigoplus_{(m_1,\dots,m_r)\in\N^r}
H^0\!\left(X,\sum m_iD_i\right).
\]
Finite generation is preserved after replacing each $D_i$ by a positive rational multiple \cite[Lemma 3.1]{KKL}.

We denote by $\mathrm{Big}(X)$ the big cone in $N^1(X)_{\R}$, by $\overline{\Eff}(X)$ the pseudoeffective cone, and by $\overline{\Mov}(X)$ the closure of the cone generated by movable divisor classes.

A projective $\Q$-factorial variety $X$ is a \emph{Mori dream space} if $\Pic(X)_{\Q}=N^1(X)_{\Q}$, the nef cone is generated by finitely many semiample classes, and there are finitely many small $\Q$-factorial modifications $\varphi_j\colon X\dashrightarrow X_j$ such that
\[
\overline{\Mov}(X)=\bigcup_j \varphi_j^{\ast}\Nef(X_j).
\]

For a rational polyhedral cone $\mathcal C\subset\Div_{\R}(X)$, we use the divisorial ring $R(X,\mathcal C)$ as in \cite[Section 3]{KKL}. Every cone considered below satisfies $\mathcal C\setminus\{0\}\subset\Bigdiv(X)$.

\section{Main results}\label{sect-3}
\subsection{Asymptotic multiplier ideal sheaves associated to tuples}
\begin{definition}
Let $(X,\Delta)$ be a projective klt pair, and let $\mathbf D=(D_1,\ldots,D_r)$ be a tuple of big $\mathbb Q$-Cartier $\mathbb Q$-divisors on $X$. For every positive integer $m$, define
\[
\mathfrak a_m(\mathbf D)
\coloneqq
\begin{cases}
\displaystyle
\prod_{j=1}^r \mathfrak b(|mD_j|),
&
\begin{array}{l}
\text{if every $mD_j$ is Cartier and every \(|mD_i|\) is nonempty},
\end{array}
\\[4mm]
(0),& \text{ otherwise.}
\end{cases}
\]
The inclusions
\[
\mathfrak b(|mD_j|)\mathfrak b(|nD_j|)
\subseteq
\mathfrak b(|(m+n)D_j|)
\]
show that $\mathfrak a_\bullet(\mathbf D)$ is a graded sequence of ideals.

For positive real numbers \(\lambda_1,\dots,\lambda_r\),
we define
\[
\mathcal J(X,\Delta;\lambda_1\|D_1\|,\dots,\lambda_r\|D_r\|)\coloneqq\max_{m}
\mathcal J\left(X,\Delta;\prod_{j=1}^r \mathfrak b(|mD_j|)^{\frac{\lambda_j}{m}}\right),
\]
where \(m\) runs through sufficiently divisible positive integers. By the Noetherian property, the collection has a maximum. We write
\[
\mathcal J(X,\Delta;c\|\mathbf{D}\|)\coloneqq \mathcal J\left(X,\Delta;c\|D_1\|,\dots,c\|D_r\|\right)
\]
for $c>0$.
\end{definition}

\begin{lemma}\label{lem-eq}
    For every \(\nu\in \Val_X\), we have $\nu(\mathfrak{a}_{\bullet}(\mathbf{D}))=\sum_{i=1}^r \sigma_{\nu}(D_i)$. In particular, we have
    \[
    \lct(X,\Delta;\mathfrak{a}_{\bullet}(\mathbf{D}))=\lct_{\sigma}(X,\Delta,\mathbf{D}).
    \]
\end{lemma}
\begin{proof}
    The valuation of a product of ideals is the sum of the valuations. Computing along a common sufficiently divisible subsequence gives
    \[
    \nu(\mathfrak{a}_{\bullet}(\mathbf{D}))=\sum_{i=1}^r \nu(\|D_i\|)=\sum_{i=1}^r \sigma_{\nu}(D_i).
    \]
    Hence, we complete the proof.
\end{proof}

\begin{proposition}\label{prop:sum}
Let $(X,\Delta)$ be a projective klt pair, and let $\mathbf D=(D_1,\ldots,D_r)$ be a tuple of big $\mathbb Q$-Cartier $\mathbb Q$-divisors on $X$. Then
\[
\mathcal J(X,\Delta;\|\mathbf D\|)=\bigcup_{\substack{0\leq D_j'\sim_{\mathbb Q}D_j\\1\leq j\leq r}}
\mathcal J\left(X,\Delta+\sum_{j=1}^rD_j'\right).
\]
More precisely, $\mathcal J(X,\Delta;\|\mathbf D\|)$ is the maximum element.
\end{proposition}

\begin{proof}
Fix effective $\mathbb Q$-divisors $D_j'\sim_{\mathbb Q}D_j$. Choose a sufficiently divisible integer $m>1$ such that every $mD_j'$ is an integral Cartier divisor. Set $H_j\coloneqq mD_j'\in |mD_j|$. Then we have 
\[
\prod_{j=1}^r\mathcal O_X(-H_j)\subseteq \prod_{j=1}^r\mathfrak b(|mD_j|) = \mathfrak a_m(\mathbf D).
\]
The monotonicity of multiplier ideals gives
\begin{align*}
\mathcal J\left(X,\Delta+\sum_{j=1}^rD_j'\right)
&=
\mathcal J\left(X,\Delta;\left(\prod_{j=1}^r\mathcal O_X(-H_j)\right)^{1/m}\right)\\
&\subseteq
\mathcal J\left(X,\Delta;\mathfrak a_m(\mathbf D)^{1/m}\right)\\
&\subseteq
\mathcal J(X,\Delta;\|\mathbf D\|).
\end{align*}

For the reverse inclusion, choose a sufficiently divisible integer
$m>1$ such that
\[
\mathcal J(X,\Delta;\|\mathbf D\|) = \mathcal J\left(X,\Delta;\mathfrak a_m(\mathbf D)^{1/m}\right).
\]
Let \(f\colon Y\to X\) be a common log resolution of $(X,\Delta)$ and the ideals $\mathfrak b(|mD_j|)$. Write \(\mathfrak b(|mD_j|)\mathcal O_Y = \mathcal O_Y(-F_{j,m})\) and \(f^*|mD_j| = |M_{j,m}|+F_{j,m}\), where $|M_{j,m}|$ is basepoint free. Set $F_m\coloneqq\sum_{j=1}^rF_{j,m}$. Then we have
\[
\mathcal J(X,\Delta;\|\mathbf D\|)
= f_*\mathcal O_Y\left(K_Y-\left\lfloor f^*(K_X+\Delta)+\frac{1}{m}F_m\right\rfloor\right).
\]

Choose general members $H_j\in |mD_j|$, and after replacing the resolution if necessary, one may write \(f^*H_j=M_{j,H_j}+F_{j,m}\), where
\[
\Supp\left(f_*^{-1}\Delta+\Exc(f)+F_m+\sum_{j=1}^rM_{j,H_j}\right)
\]
has simple normal crossings. We choose the divisors $M_{j,H_j}$ without common components and without any component in the support of
\[
f^*(K_X+\Delta)+\frac{1}{m}F_m.
\]
Since $m>1$, every component of $M_{j,H_j}$ occurs with coefficient
$\frac{1}{m}<1$. Consequently,
\[
\left\lfloor
f^*(K_X+\Delta)
+\frac{1}{m}F_m
+\frac{1}{m}\sum_{j=1}^rM_{j,H_j}
\right\rfloor
=
\left\lfloor
f^*(K_X+\Delta)+\frac{1}{m}F_m
\right\rfloor.
\]
It follows that
\begin{align*}
\mathcal J(X,\Delta;\|\mathbf D\|)
&=
\mathcal J\left(X,\Delta+\frac{1}{m}\sum_{j=1}^rH_j\right).
\end{align*}
The divisors $\frac{1}{m}H_j$ are effective and $\mathbb Q$-linearly equivalent to $D_j$. The ideal $\mathcal J(X,\Delta;\|\mathbf D\|)$ therefore belongs to the collection.
\end{proof}

We first establish a Nadel type vanishing theorem for asymptotic multiplier ideals associated with tuples of big divisors.

\begin{proposition}\label{prop;Nadel type}
Let $(X,\Delta)$ be a projective klt pair, and let $D_1,\cdots,D_r$ be big $\Q$-Cartier $\Q$-divisors on $X$. Let $\lambda_1,\dots,\lambda_r$ be positive real numbers, and let $L$ be a Cartier divisor satisfying
$$L-(K_X+\Delta)\sim_{\R}\sum_{i=1}^r \lambda_iD_i+A$$
for an ample $\R$-Cartier $\R$-divisor $A$ on $X$. Then we have
$$ H^q(X,\mathcal{O}_X(L)\otimes \mathcal{J}(X,\Delta,\lambda_1\|D_1\|,\cdots,\lambda_r\|D_r\|))=0$$
for every \(q>0\).
\end{proposition}

\begin{proof}
Fix a sufficiently divisible positive integer $m$ such that every $mD_i$ is Cartier, and
$$ \mathcal{J}(X,\Delta,\lambda_1\|D_1\|,\cdots,\lambda_r \|D_r\|)=\mathcal{J}\left(X,\Delta,\mathfrak{b}\left(|mD_1|\right)^{\frac{\lambda_1}{m}}\cdots \mathfrak{b}(|mD_r|)^{\frac{\lambda_r}{m}}\right).$$
Let $f\colon Y\to X$ be a common log resolution of $(X,\Delta)$ and the base ideals $\mathfrak{b}(|mD_i|)$. Write $f^*|mD_i|=|M_{m,i}|+F_{m,i}$, with $|M_{m,i}|$ basepoint free, and set $F_{\mathbf{\lambda}}\coloneqq \sum_i \frac{\lambda_i}{m}F_{m,i}$. Let $T\coloneqq f^{\ast}(K_X+\Delta)+F_{\mathbf{\lambda}}$. Then 
\begin{equation}\label{eq:J}
    \mathcal{J}(X,\Delta;\lambda_1|D_1|,\dots,\lambda_r|D_r|)=f_{\ast}\OO_Y(K_Y-\floor{T}).
\end{equation}
The fractional part $\{T\}\coloneqq T-\floor{T}$ has simple normal crossings support and coefficients in $[0,1)$. For the integral divisor $B\coloneqq K_Y+f^{\ast}L-\floor{T}$, we have
\begin{equation}\label{eq:B}
    B-(K_Y+\{T\})=f^{\ast}L-T\sim_{\R} f^{\ast}A+\sum_i\frac{\lambda_i}{m}M_{m,i}.
\end{equation}
The divisor on the right of \eqref{eq:B} is nef and big. Thus, Kawamata--Viehweg vanishing gives $H^q(Y,\OO_Y(B))=0$ for $q>0$, see \cite[Theorem 3.2.9]{Fuj17}. The divisor $-T$ is numerically equivalent to $\sum_i \frac{\lambda_i}{m}M_{m,i}$ over $X$, and is therefore nef over $X$. Since $f$ is birational, $-T$ is also big over $X$. The relative Kawamata--Viehweg vanishing theorem gives $R^qf_{\ast}\OO_Y(K_Y-\floor{T})=0$ for $q>0$ by \cite[Theorem 3.2.9]{Fuj17} again. The projection formula, \eqref{eq:J}, and the Leray spectral sequence give the assertion.
\end{proof}

We extend the definition of asymptotic multiplier ideals associated with tuples to pseudoeffective divisors. Let \(A\) be an ample Cartier divisor and set \(D_{i,\ell}\coloneqq D_i+\frac{1}{\ell !}A\). The following stabilization is proved by applying Proposition \ref{prop;Nadel type} and Castelnuovo--Mumford regularity, as in \cite[Proposition 3.10 and Lemma 3.11]{CJKL}.

\begin{lemma}\label{lem;decrease}
Let $(X,\Delta)$ be a projective klt pair, $D_1,\cdots,D_r$ pseudoeffective $\Q$-Cartier $\Q$-divisors on $X$, $\lambda_1,\cdots,\lambda_r>0$ real numbers, and let $A$ be an ample Cartier divisor on $X$. Then the sequence
\[
\mathcal{J}_{\ell}^A\coloneqq \mathcal{J}\left(X,\Delta;\lambda_1\|D_{1,\ell}\|,\cdots,\lambda_r\|D_{r,\ell}\|\right)
\]
is eventually decreasing and constant. Its stable value is independent of the choice of \(A\).
\end{lemma}

\begin{proof}
Let $A'$ be another ample Cartier divisor on $X$. For every sufficiently large \(\ell\), the divisor \(\ell A-A'\) is very ample. Choose a sufficiently divisible positive integer \(m\), and set \(N\coloneqq m\ell!\). For every \(j\), the image of 
\[
H^0(X,ND_j+mA')\otimes H^0(X,m(\ell A-A'))\longrightarrow H^0(X,ND_j+m\ell A)
\]
defines a sublinear series of \(|ND_j+m\ell A|\). Since \(m(\ell A-A')\) is globally generated, the base ideal of this sublinear series is \(\mathfrak{b}(|ND_j+mA'|)\). Hence, 
\[
\mathfrak{b}(|ND_j+mA'|)\subset \mathfrak{b}(|ND_j+m\ell A|).
\]
Taking products over \(j\), we obtain 
\begin{equation}\label{incl}
    \mathcal{J}_{\ell}^{A'}\subseteq \mathcal{J}_{\ell-1}^A.
\end{equation} 
By taking \(A'=A\), we have the eventual decreasing property.

Choose an ample Cartier divisor \(B\) on \(X\) such that 
\[
B-\left(K_X+\Delta+\sum_{j=1}^r \lambda_j(D_j+A)\right)
\]
is ample. Choose a very ample Cartier divisor \(H\) on \(X\). Let $d\coloneqq \dim X$. For every \(\ell\), and every \(0<i\le d\), Proposition \ref{prop;Nadel type} gives
\[
H^i(X,\mathcal{J}_{\ell}^A \otimes \OO_X(B+(d+1-i)H))=0.
\]
By the Castelnuovo--Mumford regularity, \(\mathcal{J}_{\ell}^A \otimes \OO_X(B+(d+1)H)\) is globally generated. After discarding finitely many terms, the inclusion \eqref{incl} gives a descending chain
\[
H^0(X,\mathcal{J}_{\ell+1}^A \otimes \OO_X(B+(d+1)H))\subseteq H^0(X,\mathcal{J}_{\ell}^A \otimes \OO_X(B+(d+1)H))
\]
inside the finite dimensional vector space \(H^0(X,\OO_X(B+(d+1)H))\). The descending chain of finite dimensional spaces of sections therefore stabilizes. Since each sheaf $\mathcal{J}_{\ell}^A\otimes \OO_X(B+(d+1)H)$ is globally generated, equality of these spaces of sections for all sufficiently large $\ell$ implies $\mathcal{J}_{\ell+1}^A=\mathcal{J}_{\ell}^A$ after tensoring by $\OO_X(-B-(d+1)H)$.

Interchanging \(A\) and \(A'\) shows that the two stable ideals obtained from $A$ and $A'$ contain one another. Therefore, the stable ideal is independent of the ample divisor.
\end{proof}

\begin{definition}
Let $(X,\Delta)$ be a projective klt pair, let $D_1,\dots,D_r$ be pseudoeffective $\Q$-Cartier $\Q$-divisors, and let $A$ be an ample Cartier divisor on $X$. Fix positive real numbers $\lambda_1,\dots,\lambda_r$.
\begin{itemize}
\item[(1)] We denote by $\mathcal{J}_-(X,\Delta,\lambda_1\|D_1\|,\dots,\lambda_r\|D_r\|)$ the stable element of
$$ \left\{\mathcal{J}\left(X,\Delta,\lambda_1\|D_{1,\ell}\|,\dots,\lambda_r\|D_{r,\ell}\|\right)\right\}_{\ell\in \Z_{>0}}.$$
The stable element exists and is independent of $A$ by Lemma \ref{lem;decrease}.
\item[(2)] The ideals $\mathcal{J}_-(X,\Delta,\lambda'_1\|D_1\|,\cdots,\lambda'_r\|D_r\|)$ with $\lambda_i'>\lambda_i$ form a directed collection under inclusion. The Noetherian property gives a maximum, denoted by 
\[
\mathcal{J}_{\sigma}(X,\Delta,\lambda_1\|D_1\|,\dots,\lambda_r\|D_r\|).
\]
We call the maximum the \emph{diminished multiplier ideal}. 
\end{itemize}
We use $\mathcal{J}_{-}(X,\Delta;c\|\mathbf{D}\|)$ and $\mathcal{J}_{\sigma}(X,\Delta;c\|\mathbf{D}\|)$ when all coefficients equal $c$.
\end{definition}

The following continuity result will be used in the proof of Theorem \ref{thrm:1} to pass to a limit of valuations and in the proof of Corollary \ref{coro2} to approximate a quasi-monomial valuation by divisorial valuations.

\begin{lemma}\label{lem:div}
    Let $(X,\Delta)$ be a projective klt pair, and let $\mathbf{D}$ be a tuple of pseudoeffective $\Q$-Cartier $\Q$-divisors. For every $M>0$, the function $\sigma_{\mathbf{D}}$ is finite and continuous on the set $\{\nu\in\Val_X\mid A_{X,\Delta}(\nu)\leq M\}$. Moreover,
    \[
    \lct_{\sigma}(X,\Delta,\mathbf{D})=\inf_E \frac{A_{X,\Delta}(E)}{\sigma_{\mathbf{D}}(E)},
    \]
    where $E$ runs through prime divisors over $X$. 
\end{lemma}
\begin{proof}
    We first consider a pseudoeffective $\Q$-divisor $L$ on a smooth projective variety $Y$. Let $\mathfrak{b}_0=\OO_Y$ and for $t>0$, let $\mathfrak{b}_t\coloneqq \mathcal{J}_{\sigma}(L)$ be the diminished ideal of \cite[Definition 6.2]{Leh}. Fix an ample Cartier divisor $H$ on $Y$. Then by \cite[Theorem 4.2 and Definitions 4.3 and 6.2]{Leh}, $\mathfrak{b}_t$ is a nonzero coherent ideal and, for every $a>t$ sufficiently close to $t$, we have
    \begin{equation}\label{eq:b}
        \mathfrak{b}_t=\mathcal{J}_{-}(aL)=\mathcal{J}(a(L+\varepsilon H))
    \end{equation}
    for all sufficiently small rational $\varepsilon>0$.

    For $s,t>0$, choose $\eta>0$ and then a rational number $\varepsilon>0$ sufficiently small such that $\mathfrak{b}_u=\mathcal{J}((1+\eta)u(L+\varepsilon H))$ for $u\in \{s,t,s+t\}$. By subadditivity, we have $\mathfrak{b}_{s+t}\subseteq \mathfrak{b}_s\mathfrak{b}_t$, see \cite[Section 2.3]{JM12}. 

    Fix a nontrivial $\nu\in \Val_Y$ with $A_Y(\nu)<\infty$. For $t>0$, choose a rational number $a>t$ for which \eqref{eq:b} holds. By \cite[Proposition 4.7]{Leh}, we have
    \begin{equation}\label{eq:a}
        a\sigma_{\nu}(L)\leq \nu(\mathfrak{b}_t)+A_Y(\nu).
    \end{equation}
    In particular, $\sigma_{\nu}(L)$ is finite. For the graded sequence of base ideals of $L+\varepsilon H$, \cite[Lemma 2.6 and Proposition 6.2]{JM12} give
    \[
    \nu(\mathcal{J}(a(L+\varepsilon H)))\leq a\sigma_{\nu}(L+\varepsilon H).
    \]
    By \eqref{eq:b}, letting $\varepsilon\to 0^+$ and then $a\downarrow t$ through rational numbers yields $\nu(\mathfrak{b}_t)\leq t\sigma_{\nu}(L)$. Since $a>t$, \eqref{eq:a} also gives $\nu(\mathfrak{b}_t)>t\sigma_{\nu}(L)-A_Y(\nu)$ when $\sigma_{\nu}(L)>0$. When $\sigma_{\nu}(L)=0$, this strict inequality follows from $\nu(\mathfrak{b}_t)\geq 0$ and $A_Y(\nu)>0$. Therefore, we have
    \[
    \nu(\mathfrak{b}_{\bullet})\coloneqq \lim_{t\to\infty} \frac{\nu(\mathfrak{b}_t)}{t}=\sigma_{\nu}(L).
    \]
    For every prime divisor $F$ over $Y$ and every $t>0$, we therefore obtain 
    \[
    \ord_F(\mathfrak{b}_t)>t\ord_F(\mathfrak{b}_{\bullet})-A_Y(F).
    \]
    Thus, $\mathfrak{b}_{\bullet}$ has controlled growth as in \cite[Definition 2.9]{JM12}. By \cite[Corollary 6.6]{JM12}, the function $\nu\mapsto \sigma_{\nu}(L)$ is continuous on every subset of $\Val_Y$ on which $A_Y$ is bounded.

    Now, let $f\colon Y\to X$ be a log resolution of $(X,\Delta)$, and write $K_Y+\Delta_Y=f^{\ast}(K_X+\Delta)$. Let $B$ be the reduced divisor with support $\Supp(\Delta_Y)$. Choose $0<c<1$ such that $\Delta_Y\leq (1-c)B$. Since $(Y,B)$ is log canonical, we have
    \[
    A_{X,\Delta}(\nu)=A_Y(\nu)-\nu(\Delta_Y)\geq A_Y(\nu)-(1-c)\nu(B)\geq cA_Y(\nu).
    \]
    Note that the morphism $f$ induces a homeomorphism $\Val_Y\to \Val_X$, and $\sigma_{\nu}(f^{\ast}D_i)=\sigma_{\nu}(D_i)$ for every $i$. On the set $\{\nu\in \Val_X\mid A_{X,\Delta}(\nu)\leq M\}$, we have $A_Y(\nu)\leq \frac{M}{c}$. By the smooth case, each function $\nu\mapsto \sigma_{\nu}(D_i)$ is finite and continuous on this set. Their finite sum $\sigma_{\mathbf{D}}$ is therefore finite and continuous.

    For the last assertion, fix $\nu\in \Val_X^{\ast}$ with $A_{X,\Delta}(\nu)<\infty$ and $\sigma_{\mathbf{D}}>0$. By \cite[Lemma 2.3]{Leh}, there is a net of divisorial valuations $\nu_{\alpha}$ on $Y$ such that $\nu_{\alpha}\to \nu$ and $A_Y(\nu_{\alpha})\to A_Y(\nu)$. Since $\Delta_Y$ is a fixed $\Q$-divisor on $Y$, the function $\nu\mapsto \nu(\Delta_Y)$ is continuous. Thus, $A_{X,\Delta}(\nu_{\alpha})\to A_{X,\Delta}(\nu)$, and there exists $M>0$ such that $A_{X,\Delta}(\nu_{\alpha})\leq M$ for all sufficiently large $\alpha$. Therefore, the continuity of $\sigma_{\mathbf{D}}$ gives
    \[
    \frac{A_{X,\Delta}(\nu_{\alpha})}{\sigma_{\mathbf{D}}(\nu_{\alpha})}\longrightarrow \frac{A_{X,\Delta}(\nu)}{\sigma_{\mathbf{D}}(\nu)}.
    \]
    Each $\nu_{\alpha}$ is of the form $\nu_{\alpha}=c_{\alpha}\ord_{E_{\alpha}}$ for some prime divisor $E_{\alpha}$ over $X$ and some real number $c_{\alpha}>0$. Since both $A_{X,\Delta}$ and $\sigma_{\mathbf{D}}$ are homogeneous of degree one, we have
    \[ 
    \frac{A_{X,\Delta}(\nu_\alpha)}{\sigma_{\mathbf D}(\nu_\alpha)}=\frac{A_{X,\Delta}(E_\alpha)}{\sigma_{\mathbf D}(E_\alpha)}.
    \]
    Therefore, we obtain $\inf_E\frac{A_{X,\Delta}(E)}{\sigma_{\mathbf D}(E)}\leq \frac{A_{X,\Delta}(\nu_\alpha)}{\sigma_{\mathbf D}(\nu_\alpha)}$ for every $\alpha$, and by taking the limit, we have $\inf_E\frac{A_{X,\Delta}(E)}{\sigma_{\mathbf D}(E)}\leq \frac{A_{X,\Delta}(\nu)}{\sigma_{\mathbf D}(\nu)}$.
    Since $\nu$ is arbitrary, we therefore have 
    \[
    \inf_E \frac{A_{X,\Delta}(E)}{\sigma_{\mathbf{D}}(E)}\leq \lct_{\sigma}(X,\Delta,\mathbf{D}).
    \]

    Conversely, for every prime divisor $E$ over $X$, the divisorial valuation $\ord_E$ has finite log discrepancy. Hence, $\lct_{\sigma}(X,\Delta,\mathbf{D})\leq \frac{A_{X,\Delta}(E)}{\sigma_{\mathbf{D}}(E)}$. By taking the infimum over all prime divisors $E$ over $X$, we obtain
    \[
    \lct_{\sigma}(X,\Delta,\mathbf{D})\leq \inf_E \frac{A_{X,\Delta}(E)}{\sigma_{\mathbf{D}}(E)},
    \]
    and this completes the proof.
\end{proof}

\begin{lemma}\label{lem:decreasing}
    Let $\mathbf{D}=(D_1,\dots,D_r)$ be a tuple of pseudoeffective $\Q$-Cartier $\Q$-divisors on a projective klt pair $(X,\Delta)$. Fix an ample Cartier divisor $A$, and let $\mathbf{D}_{\ell}\coloneqq (D_1+\frac{A}{\ell!},\dots,D_r+\frac{A}{\ell!})$. Then
    \[
    \lct_{\sigma}(X,\Delta,\mathbf{D}_{\ell})\downarrow\lct_{\sigma}(X,\Delta,\mathbf{D}).
    \]
\end{lemma}
\begin{proof}
    For each $\nu$, $\sigma_{\mathbf{D}_{\ell}}(\nu)$ increases to $\sigma_{\mathbf{D}}(\nu)$. As in the argument of \cite[Lemma 2.11]{CJKL}, we obtain the assertion.
\end{proof}

\begin{proposition} \label{thickening}
Let $(X,\Delta)$ be a projective klt pair, and let $\mathbf{D}=(D_1,\cdots,D_r)$ be a tuple of pseudoeffective \(\Q\)-Cartier $\Q$-divisors. Let \(\lambda\coloneqq \lct_{\sigma}(X,\Delta,\mathbf D)\), and assume that $0<\lambda<\infty$. Fix an ample Cartier divisor $A$, and let $\mathbf{D}_{\ell}\coloneqq (D_1+\frac{A}{\ell!},\dots,D_r+\frac{A}{\ell!})$ and $\lambda_{\ell}\coloneqq \lct_{\sigma}(X,\Delta,\mathbf{D}_{\ell})$. Let $x$ be the generic point of an irreducible component of $Z\left(\mathcal{J}_{\sigma}\left(X,\Delta,\lambda\|\mathbf{D}\|\right)\right)$.

Then there exist $\varepsilon_0>0$, a positive integer $p$, and graded sequences $\mathfrak{c}_{\bullet,\ell}$ of ideals in $\mathcal{O}_{X,x}$ such that, for every $0<\varepsilon'_0<\frac{\varepsilon_0}{2}$, there exists an integer $\ell_0=\ell_0(\varepsilon_0')$ such that the following assertions hold for all $\ell\ge \ell_0$:
\begin{itemize}
    \item[(1)] $\lambda_{\ell}\leq \mathrm{lct}(X_x,\Delta_x,\mathfrak{c}_{\bullet,\ell})\le (1+\varepsilon'_0)\lambda_{\ell}$,
    \item[(2)] $\mathfrak{m}^p_x\subseteq \mathfrak{c}_{1,\ell}$, and
    \item[(3)] for every \(\nu\in \Val_{X,x}\), 
    $$\nu(\mathfrak{c}_{\bullet,\ell})=\min\left\{\sigma_{\mathbf{D}_{\ell}}(\nu),p\nu(\mathfrak{m}_x)\right\}.$$
\end{itemize}
The integer $p$ is independent of $\varepsilon_0'$ and $\ell$.
\end{proposition}
\begin{proof}
We follow \cite[Proof of Proposition 3.15]{CJKL}. The proof has four steps.

\noindent \textbf{Step 1.} Let $\mathcal{I}\coloneqq \mathcal{J}_{\sigma}(X,\Delta;\lambda\|\mathbf{D}\|)$ and $\mathcal{J}_{\ell}(t)\coloneqq \mathcal{J}(X,\Delta;t\|\mathbf{D}_{\ell}\|)$. By Lemma \ref{lem:decreasing}, we have $\lambda_{\ell}\downarrow \lambda$ and hence $\lambda_{\ell}<\infty$ for sufficiently large $\ell$. For any $t>\lambda$, we have $\mathcal{J}_{\ell}(t)\neq \OO_X$ for sufficiently large $\ell$ by Lemma \ref{lem-eq} and \cite[Theorem 3.7]{CJKL}. Thus, $\mathcal{J}_{-}(X,\Delta;t\|\mathbf{D}\|)\neq \OO_X$ and $\mathcal{I}$ is nonzero and proper.

Choose $\varepsilon_0>0$ such that
\begin{equation}\label{eq:I}
    \mathcal{I}=\mathcal{J}_{-}(X,\Delta;(1+2\varepsilon_0)\lambda\|\mathbf{D}\|).
\end{equation}
Fix $0<\varepsilon_0'<\frac{\varepsilon_0}{2}$. For all sufficiently large $\ell$ we have
\begin{equation}\label{eq:1plus}
    (1+\varepsilon)\lambda_{\ell}\leq (1+2\varepsilon_0)\lambda
\end{equation}
for $\varepsilon\in [\varepsilon_0',\varepsilon_0]$. 
We may also assume that
\begin{equation}\label{eq:Jl}
    \mathcal{J}_{\ell}((1+\varepsilon_0')\lambda)=\mathcal{J}_{-}(X,\Delta;(1+\varepsilon_0')\lambda\|\mathbf{D}\|).
\end{equation}
Thus, by \eqref{eq:I}, \eqref{eq:1plus}, and \eqref{eq:Jl}, we obtain
\begin{align*}
    \mathcal{I}&\subseteq \mathcal{J}_{-}(X,\Delta;(1+\varepsilon)\lambda_{\ell}\|\mathbf{D}\|)\\
    &\subseteq \mathcal{J}_{\ell}((1+\varepsilon)\lambda_{\ell})\\
    &\subseteq \mathcal{J}_{\ell}((1+\varepsilon_0')\lambda)\\
    &=\mathcal{J}_{-}(X,\Delta;(1+\varepsilon_0')\lambda\|\mathbf{D}\|)\subseteq \mathcal{I},
\end{align*}
where the last inclusion follows from the definition of $\mathcal{J}_{\sigma}$. Hence, for all $\ell\geq \ell_0'(\varepsilon_0')$ and all $\varepsilon\in [\varepsilon_0',\varepsilon_0]$, we obtain
\begin{equation}\label{eq:step1concl}
    \mathcal{I}=\mathcal{J}_{\ell}((1+\varepsilon)\lambda_{\ell}).
\end{equation}

\smallskip

\noindent \textbf{Step 2.} Let $R\coloneqq \OO_{X,x}$ and $\mathfrak{b}_{m,\ell}\coloneqq \mathfrak{a}_m(\mathbf{D}_{\ell})R$. In particular, we have $\mathfrak{b}_{0,\ell}=R$. Write $\alpha_{\ell}\coloneqq \lct(X_x,\Delta_x,\mathfrak{b}_{\bullet,\ell})$. Since $\mathcal{I}R\neq R$, by \eqref{eq:step1concl}, \cite[Theorem 3.7 and Lemma 3.14]{CJKL} and Lemma \ref{lem-eq}, we have
\begin{equation}\label{eq:step2concl}
\lambda_{\ell}\leq \alpha_{\ell}\leq (1+\varepsilon_0')\lambda_{\ell}.    
\end{equation}
In particular, $0<\alpha_{\ell}<\infty$.

\noindent \textbf{Step 3.} The choice of $x$ implies that $\mathcal{I}R$ is $\mathfrak{m}_x$-primary. Choose an integer $n>0$ such that $\mathfrak{m}_x^n\subseteq \mathcal{I}R$, and let $\beta_{\ell}\coloneqq \lct^{\mathfrak{m}_x^n}(X_x,\Delta_x,\mathfrak{b}_{\bullet,\ell})$. By \eqref{eq:step1concl} and \cite[Theorem 3.7]{CJKL}, we obtain $\beta_{\ell}\geq (1+\varepsilon_0)\lambda_{\ell}$. Together with \eqref{eq:step2concl}, this gives
\begin{equation}
    \beta_{\ell}-\alpha_{\ell}\geq \frac{\varepsilon_0\lambda_{\ell}}{2}\geq \frac{\varepsilon_0\lambda}{2}.
\end{equation}
Choose $p\coloneqq \ceil{\frac{2n}{\varepsilon_0\lambda}}+1$. Then $p>\frac{n}{\beta_{\ell}-\alpha_{\ell}}$ for every $\ell\geq\ell_0'(\varepsilon_0')$, and $p$ is independent of $\varepsilon_0'$.

\smallskip

\noindent \textbf{Step 4.} Define ideals of $R$ by
\[
\mathfrak{c}_{m,\ell}\coloneqq \sum_{i=0}^m \mathfrak{b}_{i,\ell}\mathfrak{m}_x^{p(m-i)}.
\]
Note that $\mathfrak{c}_{\bullet,\ell}$ is graded and the first term gives $\mathfrak{m}_x^{pm}\subseteq \mathfrak{c}_{m,\ell}$. For a valuation $\nu$ centered on $X_x$, we therefore obtain that
\[
\nu(\mathfrak{c}_{m,\ell})=\min_{0\leq i\leq m}\{\nu(\mathfrak{b}_{i,\ell})+p(m-i)\nu(\mathfrak{m}_x)\}.
\]
Since $\nu(\mathfrak{b}_{i,\ell})\geq i\nu(\mathfrak{b}_{\bullet,\ell})$, we obtain $\frac{\nu(\mathfrak{c}_{m,\ell})}{m}\geq \min\{\nu(\mathfrak{b}_{\bullet,\ell}),p\nu(\mathfrak{m}_x)\}$. The first and last terms give the reverse inequality, and Lemma \ref{lem-eq} proves the assertion (3).

Note that the local thresholds of $\mathfrak{b}_{\bullet,\ell}$ and $\mathfrak{c}_{\bullet,\ell}$ are equal. Choose, by \cite[Theorem 1.1]{Xu20}, a valuation $\omega$ computing $\alpha_{\ell}$, normalized by $\omega(\mathfrak{b}_{\bullet,\ell})=1$.  Now, from the definition of $\beta_{\ell}$, we have $\beta_{\ell}\leq A_{X,\Delta}(\omega)+n\omega(\mathfrak{m}_x)=\alpha_{\ell}+n\omega(\mathfrak{m}_x)$. Thus, $p\omega(\mathfrak{m}_x)>1$ and assertion (3) gives $\omega(\mathfrak{c}_{\bullet,\ell})=1$. Since $\mathfrak{b}_{m,\ell}\subseteq \mathfrak{c}_{m,\ell}$, we conclude that
\begin{equation}\label{eq:step4concl}
    \alpha_{\ell}\leq \lct(X_x,\Delta_x,\mathfrak{c}_{\bullet,\ell})\leq A_{X,\Delta}(\omega)=\alpha_{\ell}.
\end{equation}
Combining \eqref{eq:step4concl} with \eqref{eq:step2concl} gives the assertion (1).
\end{proof}

\subsection{Log canonical thresholds associated to tuples}

In this subsection, we give proofs of Theorem \ref{thrm:1} and Corollary \ref{coro}.

\begin{proof}[Proof of Theorem \ref{thrm:1}]
We adapt the argument of \cite[Proof of Theorem 1.1]{CJKL}. By Lemma \ref{lem:div}, $\lambda\coloneqq \lct_{\sigma}(X,\Delta,\mathbf{D})\geq 1$. 

\smallskip

\noindent\textbf{Step 1.} Fix an ample Cartier divisor $A$ and write $\mathbf{D}_{\ell}\coloneqq (D_1+\frac{A}{\ell!},\dots,D_r+\frac{A}{\ell!})$. Choose a generic point $x$ of an irreducible component of $Z(\mathcal{J}_{\sigma}(X,\Delta;\lambda\|\mathbf{D}\|))$. The proof of Proposition \ref{thickening} shows that this closed subset is nonempty. Let $\varepsilon_0,p$, and $\mathfrak{c}_{\bullet,\ell}$ be as in Proposition \ref{thickening}.

Choose a sequence $\varepsilon_j\downarrow 0$ with $0<\varepsilon_j<\frac{\varepsilon_0}{2}$, and choose a strictly increasing sequence of integers $k_j\geq \ell_0(\varepsilon_j)$. Let
\[
\mu_j\coloneqq \lct(X_x,\Delta_x,\mathfrak{c}_{\bullet,k_j}),\quad \lambda_{k_j}\coloneqq \lct_{\sigma}(X,\Delta,\mathbf{D}_{k_j}).
\]
Then $\lambda_{k_j}\leq \mu_j\leq (1+\varepsilon_j)\lambda_{k_j}$, and Lemma \ref{lem:decreasing} gives $\mu_j\to \lambda$.

By \cite[Theorem 1.1]{Xu20}, choose a quasi-monomial valuation $\nu_j$ computing $\mu_j$, normalized by $A_{X,\Delta}(\nu_j)=1$. Since $\mathfrak{m}_x^{pm}\subseteq \mathfrak{c}_{m,k_j}$, any valuation centered on $X_x$ away from $x$ evaluates $\mathfrak{c}_{\bullet,k_j}$ to zero. Thus, $c_X(\nu_j)=x$, and
\begin{equation}
    \nu_j(\mathfrak{m}_x)\geq \frac{1}{p}\nu_j(\mathfrak{c}_{\bullet,k_j})=\frac{1}{p\mu_j}.
\end{equation}
Choose $C>\sup_j \mu_j$ and set $\delta\coloneqq \frac{1}{2pC}$. Then all $\nu_j$ belong to
\[
W\coloneqq \{\nu\in \Val_{X,x}\mid \nu(\mathfrak{m}_x)\geq \delta, A_{X,\Delta}(\nu)\leq 1\}.
\]

\noindent\textbf{Step 2.} 
We follow the argument of \cite[Proposition 5.9]{JM12}. 

The set $W$ is compact. For each $0\neq f\in R\coloneqq \OO_{X,x}$, we have $c_f\coloneqq \lct(X_x,\Delta_x;(f))>0$ since the given pair is klt. Every $\nu\in W$ satisfies $0\leq \nu(f)\leq \frac{1}{c_f}$. Therefore, $W$ is contained in the product $\Pi_{0\neq f\in R} [0,\frac{1}{c_f}]$.

The values are finite on every nonzero element of $R$, and therefore the limit in this product extends to a valuation of $k(X)$. Also, every $\nu\in W$ satisfies $\nu(f)\geq \delta$ for $f\in \mathfrak{m}_x\setminus \{0\}$. Hence, the center of this limit is $x$. The function $\nu\mapsto \nu(\mathfrak{m}_x)$ is continuous since $\mathfrak{m}_x$ is finitely generated. Thus, $W$ is closed in the product and is therefore compact. Choose a convergent subsequence of $(\nu_j)$, with limit $\nu_0\in W$.

\smallskip

\noindent\textbf{Step 3.} Proposition \ref{thickening} gives $0<\nu_j(\mathfrak{c}_{\bullet,k_j})\leq \sigma_{\mathbf{D}_{k_j}}(\nu_j)\leq \sigma_{\mathbf{D}}(\nu_j)$. Since $A_{X,\Delta}(\nu_j)=1$, we obtain
\[
\lambda\leq \frac{1}{\sigma_{\mathbf{D}}(\nu_j)}\leq \frac{1}{\nu_j(\mathfrak{c}_{\bullet,k_j})}=\mu_j\longrightarrow \lambda.
\]
By Lemma \ref{lem:div}, we have $\sigma_{\mathbf{D}}(\nu_0)=\frac{1}{\lambda}$. Since $0<A_{X,\Delta}(\nu_0)\leq 1$, the definition of $\lambda$ gives
\[
\lambda\leq \frac{A_{X,\Delta}(\nu_0)}{\sigma_{\mathbf{D}}(\nu_0)}\leq \lambda.
\]
Thus, $\nu_0$ computes $\lambda$.

\smallskip

\noindent\textbf{Step 4}. For $m\geq 1$, let $\mathfrak{a}_m(\nu_0)\coloneqq \{f\in R\mid \nu_0(f)\geq m\}$ and $\mathfrak{a}_0(\nu_0)=R$. Then $\nu_0(\mathfrak{a}_{\bullet}(\nu_0))=1$. 

Let $\omega$ be a valuation centered on $X_x$ with $\omega(\mathfrak{a}_{\bullet}(\nu_0))=1$. For every $f\in R$ with $\nu_0(f)>0$, we have $n\omega(f)\geq \floor{n\nu_0(f)}$. Hence, $\omega(f)\geq \nu_0(f)$ for every $f\in R$, and $\omega$ is centered at $x$, and $\sigma_{\mathbf{D}}(\omega)\geq \sigma_{\mathbf{D}}(\nu_0)$. Step 3 then implies that
\[
A_{X,\Delta}(\omega)\geq \lambda\sigma_{\mathbf{D}}(\omega)\geq \lambda\sigma_{\mathbf{D}}(\nu_0)=A_{X,\Delta}(\nu_0).
\]
After rescaling, this proves that $\lct(X_x,\Delta_x,\mathfrak{a}_{\bullet}(\nu_0))=A_{X,\Delta}(\nu_0)$.

By \cite[Theorem 1.1]{Xu20}, choose a quasi-monomial minimizer $\omega$, normalized by $\omega(\mathfrak{a}_{\bullet}(\nu_0))=1$. Since we have $\omega\geq \nu_0$ on $R$ and $A_{X,\Delta}(\omega)=A_{X,\Delta}(\nu_0)$, we obtain
\[
\lambda\leq \frac{A_{X,\Delta}(\omega)}{\sigma_{\mathbf{D}}(\omega)}\leq \frac{A_{X,\Delta}(\nu_0)}{\sigma_{\mathbf{D}}(\nu_0)}=\lambda.
\]
Thus, the quasi-monomial valuation $\omega$ computes $\lct_{\sigma}(X,\Delta,\mathbf{D})$.
\end{proof}

\begin{proof}[Proof of Corollary \ref{coro}]
Assume first that $(X,\Delta,\mathbf{D})$ is coupled pklt, and let $\lambda\coloneqq \lct_{\sigma}(X,\Delta,\mathbf{D})$. Let $\nu_0$ be a quasi-monomial minimizer by Theorem \ref{thrm:1}.

Choose a log smooth model $(Y,E)$ and a stratum of $E$ with generic point $\eta$ such that $\nu_0$ belongs to the relative interior of $\mathrm{QM}_\eta(Y,E)$. The function $\phi(\nu)\coloneqq A_{X,\Delta}(\nu)-\sigma_{\mathbf{D}}(\nu)$ is finite and homogeneous of degree one. Each $\sigma_{\nu}(D_i)$ is concave in the weights. Therefore, $\phi$ is convex and locally Lipschitz near $\nu_0$.

Choose $L>0$ and a neighborhood $U$ of $\nu_0$ in this relative interior such that $|\phi(\nu_0)-\phi(\omega)|\leq L\|\nu_0-\omega\|$ for $\omega\in U$. Since $(X,\Delta,\mathbf{D})$ is coupled pklt, there exists $a>0$ such that $A_{X,\Delta}(E)-\sigma_{\mathbf{D}}(E)\geq a$ for every prime divisor $E$ over $X$. Then we have $\phi(\ord_F)\geq a$ for a prime divisor $F$ over $X$ and $q\|\nu_0-\omega\|<\frac{a}{2L}$. Since $\phi$ is homogeneous, we have
\[
q\phi(\nu_0)\geq \phi(q\omega)-Lq\|\nu_0-\omega\|>\frac{a}{2}>0.
\]
Thus, $A_{X,\Delta}(\nu_0)>\sigma_{\mathbf{D}}(\nu_0)$ and $\lambda>1$.

Conversely, assume that $\lct_{\sigma}(X,\Delta,\mathbf{D})>1$. Choose a real number $c$ with $1<c<\lct_{\sigma}(X,\Delta,\mathbf{D})$, and choose $m>0$ such that $m(K_X+\Delta)$ is Cartier. Then $\sigma_{\mathbf{D}}(E)\leq \frac{A_{X,\Delta}(E)}{c}$ for every prime divisor $E$ over $X$. By Lemma \ref{lem-logdiscr}, we have
\[
A_{X,\Delta}(E)-\sigma_{\mathbf{D}}(E)\geq \left(1-\frac{1}{c}\right)A_{X,\Delta}(E)\geq \frac{c-1}{cm}>0.
\]
This bound is independent of $E$ and hence, $(X,\Delta,\mathbf{D})$ is coupled pklt.
\end{proof}

Corollary \ref{coro} also shows that one can choose effective $\Q$-divisors $G_i\sim_{\Q}D_i$ simultaneously so that $(X,\Delta+\sum_i G_i)$ is klt.

\begin{corollary}
    Let \((X,\Delta,\mathbf{D})\) be a coupled potentially klt tuple, where \(\mathbf D=(D_1,\dots,D_r)\). Assume that all \(D_i\) are big. Then there exist effective \(\Q\)-divisors \(G_i\sim_{\Q} D_i\) such that $(X,\Delta+\sum_{i=1}^r G_i)$ is klt.
\end{corollary}
\begin{proof}
    By Corollary \ref{coro}, we have \(\lct_{\sigma}(X,\Delta,\mathbf{D})>1\). Lemma \ref{lem-eq} and \cite[Theorem 3.7]{CJKL} give \(\mathcal{J}(X,\Delta;\|\mathbf{D}\|)=\OO_X\). By Proposition \ref{prop:sum}, we have effective divisors $G_i\sim_{\Q} D_i$ with $\mathcal{J}(X,\Delta+\sum_i G_i)=\OO_X$. Hence, $(X,\Delta+\sum_i G_i)$ is klt.
\end{proof}

\subsection{Geometric generic fibers}\label{subsect-ggf}
In this subsection, we establish Corollary \ref{coro2}. For a morphism $f\colon X\to S$ and a point $s\in S$, write $X_s\coloneqq X\times_S \Spec k(s)$. For a geometric point $\overline{s}\to S$, write $X_{\overline{s}}\coloneqq X\times_{S} \Spec \overline{k(s)}$.  Write $D_{i,s}$ and $\Delta_s$ for the corresponding restrictions.

\begin{proof}[Proof of Corollary \ref{coro2}]
Let $\nu\in \mathrm{Val}^*_{X_{\overline{\eta}}}$ be a quasi-monomial valuation computing $\mathrm{lct}_{\sigma}(X_{\overline{\eta}},\Delta_{\overline{\eta}},\mathbf{D}_{\overline{\eta}})$, and let $f\colon \left(X',E\coloneqq \sum^r_{i=1} E_i\right)\to (X_{\overline{\eta}},\Delta_{\overline{\eta}})$ be a log smooth model of $(X,\Delta)$ such that $\nu\in \mathrm{QM}(X',E)$. After generic finite base change of $S$, we can descend $\nu\in \mathrm{QM}(X',E)$ to $\nu_S\in \mathrm{QM}(X'_S,E_S)$ (cf. \cite[Remark 2.8]{KL25}).

Let $\nu_{i,S}\coloneqq \lambda_i \mathrm{ord}_{E_{i,S}}\in \mathrm{QM}(X'_S,E_S)$ be a sequence of divisorial valuations, and let $f_i\colon X'_{iS}\to X'_S$ be a toroidal blowup such that $E_{iS}$ is a prime divisor on $X'_{iS}$. The varieties and morphisms can be illustrated by the following diagram:
$$
\begin{tikzcd}
X'_{i,s}\ar[hook]{d}\ar["f_{i,s}"]{r}& (X'_{s},E_s) \ar[hook]{d}\ar{r}& (X_{s},\Delta_s)\ar[hook]{d} \\
X'_{i,S} \ar["f_i"]{r}& (X'_S,E_S) \ar{r}& (X,\Delta) \\
X'_{i,\overline{\eta}} \ar[hook]{u}\ar["f_{i,\overline{\eta}}"]{r}& (X',E) \ar[hook]{u}\ar{r}& (X_{\overline{\eta}},\Delta_{\overline{\eta}})\ar[hook]{u}
\end{tikzcd}
$$
By \cite[Theorem 1.1]{Jun26}, we obtain that
\begin{equation} \label{imrn} 
\mathrm{vol}(D_{i,\overline{\eta}})=\mathrm{vol}(D_{i,s})\text{ for }s\in S'.
\end{equation}
We claim that
\begin{equation} \label{michigan} 
\sigma_{E_{j,\overline{\eta}}}(D_{i,\overline{\eta}})\le \sigma_{E_{j,s}}(D_{i,s}) \text{ for every }i,j\text{ and }s\in S'.
\end{equation}
Indeed, if $a>\sigma_{E_{j,s}}(D_{j,s})$, then
$$
\begin{aligned}
\mathrm{vol}(f^*_{j,\overline{\eta}}D_{\overline{\eta}}-aE_{j,\overline{\eta}})&\le \mathrm{vol}(f^*_{j,s}D_{i,s}-aE_{j,s}) & (1)
\\ & <\mathrm{vol}(f^*_{i,s}D_{j,s})& (2)
\\ &=\mathrm{vol}(D_{i,s})& (3)
\\ &=\mathrm{vol}(D_{i,\overline{\eta}})& (4)
\\ &=\mathrm{vol}(f^*_{j,\overline{\eta}}D_{i,\overline{\eta}}),& (5)
\end{aligned}
$$
where
\begin{itemize}
    \item we used \cite[Theorem 12.8]{Har77} in (1),
    \item (2) follows from \cite[Proposition 2.1]{FKL16},
    \item (3), (5) is due to the invariance of volume under birational morphisms, and
    \item (4) follows from (\ref{imrn}).
\end{itemize}
Therefore, by \cite[Proposition 2.1]{FKL16} again, $a>\sigma_{E_{j\overline{\eta}}}(D_{i\overline{\eta}})$, and we proved (\ref{michigan}). Considering \cite[Corollary 2.13]{Kim25} gives $\sigma_{\mathbf{D}_{\overline{\eta}}}(\nu)\le \sigma_{\mathbf{D}_s}(\nu_s)$.

Hence, by \cite[Proposition 2.5]{Kim25},
$$ \mathrm{lct}_{\sigma}(X_{\overline{\eta}},\Delta_{\overline{\eta}},\mathbf{D}_{\overline{\eta}})=\frac{A_{X_{\overline{\eta}},\Delta_{\overline{\eta}}}(\nu)}{\sigma_{\mathbf{D}_{\overline{\eta}}}(\nu)}\ge \frac{A_{X_s,\Delta_s}(\nu_s)}{\sigma_{\mathbf{D}_s}(\nu_s)}>1,$$
where the last inequality follows from Corollary \ref{coro}. Hence, by Corollary \ref{coro} again, $(X_{\overline{\eta}},\Delta_{\overline{\eta}},\mathbf{D}_{\overline{\eta}})$ is coupled pklt.
\end{proof}

\subsection{Applications to varieties of pklt type}\label{subsect-appl}
We first reduce the applications to the $\Q$-factorial case.
\begin{lemma}\label{lem:qfactorial-reduction}
Let $(X,\Delta)$ be a projective pklt pair. There exists a projective small birational morphism \(q\colon X^{\mathrm q}\to X\) such that $X^{\mathrm q}$ is $\Q$-factorial. If $\Delta^{\mathrm q}$ is the strict transform of $\Delta$, then $(X^{\mathrm q},\Delta^{\mathrm q})$ is pklt. Moreover, for any $\Q$-Cartier $\Q$-divisors $D_1,\dots,D_r$ on $X$,
\[
R(X;D_1,\dots,D_r)
\simeq
R(X^{\mathrm q};q^*D_1,\dots,q^*D_r).
\]
\end{lemma}
\begin{proof}
    By \cite[Corollary 1.4.3]{BCHM}, there exists a projective small $\Q$-factorialization $q$. Since $q$ is small, $K_{X^{\rm q}}+\Delta^{\rm q}=q^{\ast}(K_X+\Delta)$, and the prime divisors on $X$ and $X^{\rm q}$ correspond. Both the log discrepancies and the asymptotic orders of the anticanonical divisor coincide on every prime divisor over the two varieties. Hence, $(X^{\rm q},\Delta^{\rm q})$ is pklt.
\end{proof}

The following proposition relates big divisors on a pklt pair to klt adjoint divisors. This will be used in the applications to finite generation and birational Zariski decompositions.

\begin{proposition}\label{prop:uniform-adjointization}
Let $(X,\Delta)$ be a projective pklt pair, and let $D_1,\dots,D_r$ be big $\Q$-Cartier $\Q$-divisors. Then there exist a rational number $\lambda>0$ and effective big $\Q$-Cartier $\Q$-divisors $G_1,\dots,G_r$ such that $(X,\Delta+G_i)$ is klt and $K_X+\Delta+G_i\sim_{\Q}\lambda D_i$ for every $i$.
\end{proposition}
\begin{proof}
By \cite[Proposition 4.1]{CJKL}, there exists $\varepsilon>0$ such that $(X,\Delta,-(1+\varepsilon)(K_X+\Delta))$ is a pklt triple. Lemma \ref{lem:div} then gives 
\[
a\coloneqq \sup_{\nu\in \Val_X^{\ast}} \frac{\sigma_{\nu}(-(K_X+\Delta))}{A_{X,\Delta}(\nu)}\leq \frac{1}{1+\varepsilon}<1.
\]
For each $i$, Lemma \ref{lem:sup} shows that $b_i\coloneqq \sup_{\nu} \frac{\sigma_{\nu}(D_i)}{A_{X,\Delta}(\nu)}$ is finite. Choose a rational number $\lambda>0$ such that $a+\lambda b_i<1$ for every $i$. Since the divisor $-(K_X+\Delta)+\lambda D_i$ is big, we have
\[
\frac{\sigma_{\nu}(-(K_X+\Delta)+\lambda D_i)}{A_{X,\Delta}(\nu)}\leq a+\lambda b_i<1.
\]
By Lemma \ref{lem-eq}, \cite[Theorem 3.7]{CJKL} and Proposition \ref{prop:sum}, there exists an effective divisor $G_i\sim_{\Q} -(K_X+\Delta)+\lambda D_i$ such that $(X,\Delta+G_i)$ is klt. 
\end{proof}

\begin{proof}[Proof of Theorem \ref{thm:BZD}]
    Let $D$ be a big $\Q$-Cartier $\Q$-divisor on $X$. By \cite[Theorem 3.3]{KM98} and \cite[Proof of Theorem 2.4]{KL25}, there exists a good $D$-minimal model $X\dashrightarrow Y$. Let $D_Y$ be the strict transform of $D$, and let $f\colon X'\to X$ and $g\colon X'\to Y$ be a common resolution. Then, $f^*D=g^*D_Y+E$ for some effective $g$-exceptional $\Q$-Cartier $\Q$-divisor $E$ on $X'$, and this is a birational Zariski decomposition of $D$ with semiample positive part.
\end{proof}

\begin{lemma}\label{lem-bzd}
    Let \(f\colon Y\to X\) be a projective birational morphism, and suppose that $f^{\ast}D=P+N$ is a birational Zariski decomposition with semiample positive part. Then $\sigma_{\nu}(D)=\nu(N)$ for every valuation \(\nu\in \Val_X\).
\end{lemma}
\begin{proof}
    Choose a sufficiently divisible positive integer \(m\) such that \(mP\) is basepoint free and \(mN\) is integral. Since we have $H^0(Y,mP)\simeq H^0(Y,mf^{\ast}D)$,
    the base ideal of \(|mf^{\ast}D|\) is \(\OO_Y(-mN)\). The same equality holds after replacing $m$ by any sufficiently divisible positive multiple. Therefore, for every $\nu\in\Val_X$,
    \[
    \frac{1}{m}\nu\bigl(\mathfrak b(|mf^*D|)\bigr)=\nu(N).
    \]
    Since $D$ is big, we obtain
    \[
    \sigma_\nu(D)=\sigma_\nu(f^*D)=\inf_m\frac{1}{m}\nu\bigl(\mathfrak b(|mf^*D|)\bigr)=\nu(N).\qedhere
    \]
\end{proof}

\begin{proof}[Proof of Corollary \ref{cor:divisorial}]
    By Theorem \ref{thm:BZD}, each \(D_i\) admits a birational Zariski decomposition with semiample positive part. Take a common resolution \(f\colon Y\to X\) on which $f^{\ast}D_i=P_i+N_i$
    for all \(i\), where \(P_i\) is semiample and \(N_i\geq 0\). Set $N\coloneqq \sum_{i=1}^r N_i$ and $K_Y+\Delta_Y=f^{\ast}(K_X+\Delta)$. After taking a higher model, we may assume that $\Supp(\Delta_Y)\cup \Supp(N)$ has simple normal crossings. 
    
    By Lemma \ref{lem-bzd}, we have \(\sigma_{\mathbf{D}}(\nu)=\nu(N)\) for every valuation \(\nu\). Consequently,
    \[
    \lct_{\sigma}(X,\Delta,\mathbf{D})=\inf_{\nu\in \Val_Y^{\ast}}\frac{A_{Y,\Delta_Y}(\nu)}{\nu(N)}=\lct(Y,\Delta_Y;N).
    \]
    Since $\Supp(\Delta_Y)\cup\Supp(N)$ has simple normal crossings, the threshold is
    \[
    \lct_{\sigma}(X,\Delta,\mathbf D)=\min_{F\subseteq\Supp(N)}\frac{A_{Y,\Delta_Y}(F)}{\mult_F N}.
    \]
    Hence, it is rational and, when finite, is computed by a prime divisor on $Y$.
\end{proof}

\begin{proof}[Proof of Theorem \ref{thm:fg}]
Choose an effective $\Q$-divisor \(\Delta\) such that \((X,\Delta)\) is potentially klt. By Lemma \ref{lem:qfactorial-reduction}, we may assume that $X$ is $\Q$-factorial. Proposition \ref{prop:uniform-adjointization} gives $\lambda\in\Q_{>0}$ and effective big divisors $G_i$ such that $(X,\Delta+G_i)$ is klt and $K_X+\Delta+G_i\sim_{\Q}\lambda D_i$. Each boundary $\Delta+G_i$ is big, and therefore, the multigraded adjoint ring
\[
R\bigl(X;K_X+\Delta+G_1,\dots,K_X+\Delta+G_r\bigr)
\]
is finitely generated by \cite[Theorem 3.5(2)]{KKL}. The assertion follows from \cite[Lemma 3.1]{KKL}.
\end{proof}

\subsection{The boundary of the movable cone}
We conclude by showing that a projective $\Q$-factorial variety of pklt type is a Mori dream space when every nonzero class in the closed movable cone is big. The following convex geometric lemma will be used to prove the rational polyhedrality of the movable cone.

\begin{lemma}\label{lem:fdrvs}
    Let \(V_{\Q}\) be a finite dimensional $\Q$-vector space, and let \(V\coloneqq V_{\Q}\otimes_{\Q} \R\). Let $C\subseteq V$ be a closed pointed convex cone, and let \(\ell\in V_{\Q}^{\ast}\) be strictly positive on \(C\setminus \{0\}\). We also write $\ell$ for the real linear extension to $V.$ Assume that \(C\) is locally rational polyhedral at every point of
    \[
    S\coloneqq \{x\in C\mid \ell(x)=1\}.
    \]
    Then \(C\) is rational polyhedral.
\end{lemma}
\begin{proof}
    We may assume that $C\neq \{0\}$, and fix a Euclidean norm on $V$. The restriction of $\ell$ to the intersection of $C$ with the unit sphere has a minimum, which is positive value. Hence, $S$ is bounded and since $S$ is closed, $S$ is compact.

    For each $x\in S$, choose an open neighborhood $U_x\subseteq V$ and a rational polyhedral cone $C_x\subseteq V$ such that $C\cap U_x=C_x\cap U_x$. Every extreme point of $S$ contained in $U_x$ is a vertex of $C_x\cap \{\ell=1\}$. This polyhedron has only finitely many vertices, and every vertex is rational since $\ell$ is rational. A finite subcover of $\{U_x\cap S\}_{x\in S}$ therefore implies that $S$ has only finitely many extreme points, all of which belong to $V_{\Q}$. Since a compact convex set is the convex hull of its extreme points, $S$ is a rational polytope. Therefore, $C=\R_{\geq 0}S$ is rational polyhedral.
\end{proof}

\begin{proof}[Proof of Theorem \ref{thm:Mov cone}]
    Let $C\coloneqq \overline{\Mov}(X)$, and let $\pi\colon \Div_{\R}(X)\to N^1(X)_{\R}$ be the natural projection. Choose $\ell\in N^1(X)_{\Q}^{\ast}$ which is strictly positive on $\overline{\mathrm{Eff}}(X)\setminus \{0\}$, and extend $\ell$ linearly to $N^1(X)_{\R}$. Then the set $S\coloneqq \{\alpha\in C\mid \ell(\alpha)=1\}$ is compact, and is contained in $\Bigdiv(X)$.

    Fix $\alpha\in S$. Choose effective big $\Q$-divisors $D_1,\dots,D_r$ such that the image of $\mathcal{C}_{\alpha}\coloneqq \sum_{j=1}^r \R_{\geq 0}D_j$ under $\pi$ contains an open neighborhood of $\alpha$ and $\pi(\mathcal{C}_{\alpha})\setminus \{0\}\subseteq \Bigdiv(X)$. By Theorem \ref{thm:fg}, the ring $R(X;D_1,\dots,D_r)$ is finitely generated. By \cite[Corollary 3.4]{KKL}, the cone $\pi(\mathcal{C}_{\alpha})\cap C$ is a finite union of rational polyhedral cones. Since $\pi(\mathcal{C}_{\alpha})\cap C$ is convex, this intersection is also rational polyhedral. Thus, $C$ is locally rational polyhedral at $\alpha$. By Lemma \ref{lem:fdrvs}, we conclude that $C$ is rational polyhedral.

    Choose effective big \(\Q\)-divisors \(M_1,\dots,M_s\) whose numerical classes generate $C$. Let $\mathcal{C}\coloneqq \sum_{j=1}^s \R_{\geq 0}M_j$. The ring $R(X;M_1,\dots,M_s)$ is finitely generated by Theorem \ref{thm:fg}. For every nonzero $D\in \mathcal{C}\cap \Div_{\Q}(X)$ and every $\Q$-divisor $D'\equiv D$, the divisor $D'$ is big. Hence, again by Theorem \ref{thm:fg}, $R(X,D')$ is finitely generated and $D$ is gen in the sense of \cite[Definition 4.7]{KKL}. The numerical classes of the $M_j$ span $N^1(X)_{\R}$, and $\mathcal{C}$ contains an ample divisor. By \cite[Theorem 1.1]{KKL}, there are finitely many optimal models $f_i\colon X\dashrightarrow X_i$, with $X_i$ projective and $\Q$-factorial, such that every nonzero $\Q$-divisor in $\mathcal{C}$ has a nef transform on at least one $X_i$.
    
    Moreover, each $f_i$ is small. Indeed, let $D\in \mathcal{C}$ be a nonzero $\Q$-divisor for which $f_i$ is an optimal model, and set $D_i\coloneqq (f_i)_{\ast}D$. Since $D$ is big and $[D]\in C$, we have $N_{\sigma}(D)=0$. On a common resolution $p\colon W\to X$ and $q\colon W\to X_i$, write $p^{\ast}D=q^{\ast}D_i+F$, where $F\geq 0$ is $q$-exceptional. If $f_i$ contracts a prime divisor $G$ on $X$, then the strict transform of $G$ occurs in $F$ with positive coefficient. Since $D_i$ is nef and big, this coefficient equals $\sigma_G(D)$, contradicting $N_{\sigma}(D)=0$. Thus, $f_i$ is an isomorphism in codimension one. 

    Strict transform induces isomorphisms between the rational divisor class groups of $X$ and $X_i$ and preserves numerical equivalence by the negativity lemma; see \cite[Corollary 2.6]{KKL}. Since $X$ and $X_i$ are $\Q$-factorial, $\Pic(X_i)_{\Q}=N^1(X_i)_{\Q}$ follows from $\Pic(X)_{\Q}=N^1(X)_{\Q}$.

    Let $M_{j,i}$ be the strict transform of $M_j$ on $X_i$. The ring $R(X_i;M_{1,i},\ldots,M_{s,i})$ is isomorphic to $R(X;M_1,\ldots,M_s)$ and is therefore finitely generated. The numerical image of $\mathcal C_i\coloneqq\sum_{j=1}^s\R_{\ge0}M_{j,i}$ contains $\Nef(X_i)$. In particular, $\mathcal C_i$ contains an ample divisor. By \cite[Corollary 3.4]{KKL}, the cone $\Nef(X_i)$ is rational polyhedral and is generated by finitely many semiample classes. Since $\Pic(X_i)_{\Q}=N^1(X_i)_{\Q}$, every nef $\Q$-divisor on $X_i$ is semiample. Including the identity of $X$ among the $f_i$ also shows that $\Nef(X)$ is rational polyhedral and that every nef $\Q$-divisor on $X$ is semiample.

    Every rational class in $C$ has a nef transform on one of the $X_i$. The cones $f_i^*\Nef(X_i)$ are closed, and there are only finitely many such cones. Since rational classes are dense in $C$, we obtain $C\subseteq\bigcup_i f_i^*\Nef(X_i)$. The reverse inclusion follows since each $f_i$ is small and nef classes are limits of ample classes. Therefore,
    \[
    \overline{\Mov}(X)=\bigcup_i f_i^*\Nef(X_i).
    \]
    Together with $\Pic(X)_{\Q}=N^1(X)_{\Q}$, these properties show that $X$ is a Mori dream space by \cite[Definition 1.10]{HK}.
\end{proof}

\bibliographystyle{habbvr}
\bibliography{biblio}

\end{document}